\documentclass[12pt]{amsart}
\usepackage{amssymb}
\usepackage{amsmath, amscd}
\usepackage{amsthm}
\usepackage{ulem}

\newtheorem{theorem}{Theorem}[section]

\newtheorem{proposition}[theorem]{Proposition}
\newtheorem{lemma}[theorem]{Lemma}

\newtheorem{corollary}[theorem]{Corollary}
\theoremstyle{definition}

\newtheorem{example}[theorem]{Example}

\def\val#1{\vert #1 \vert}
\def\ov#1{\overline{#1}}

\def\Z{{\mathbb Z}}
\def\Q{{\mathbb Q}}
\def\N{{\mathbb N}}

\def\B{{\mathcal B}}

\usepackage[utf8]{inputenc}
\usepackage{amssymb}
\usepackage{amsmath, amscd}
\usepackage{amsthm}

\date\today

\begin{document}

\author[P.V. Danchev]{Peter V. Danchev}
\address{Institute of Mathematics and Informatics, Bulgarian Academy of Sciences \\ "Acad. G. Bonchev" str., bl. 8, 1113 Sofia, Bulgaria.}
\email{danchev@math.bas.bg; pvdanchev@yahoo.com}
\author[P.W. Keef]{Patrick W. Keef}
\address{Department of Mathematics, Whitman College, Walla Walla, WA 99362, USA}
\email{keef@whitman.edu}

\title[Strongly Hopfian abelian Groups]{On Strongly Hopfian Mixed  abelian Groups}

\maketitle

\begin{abstract} The classes of abelian groups that are (uniformly) strongly Hopfian abelian groups, and dually, (uniformly) strongly co-Hopfian abelian groups have been studied by several authors, including Abdelalim (2015) and Abdelalim-Chillali-Essanouni (2015).
This paper extends these investivations in the case of (genuinely) mixed groups. For example, it is shown that a reduced group that is (uniformly) strongly co-Hopfian will always be (uniformly) strongly Hopfian. In addition, a result of Chekhlov-Danchev (submitted) characterizing when a torsion-free group is (uniformly) strongly Hopfian is generalized to the case of (global) Warfield groups.
\end{abstract}

\def\vv#1{\vert{#1}\vert_p}
\def\vvp#1#2{\vert{#1}\vert_{#2}}

\section{Introduction and Background}

All groups considered and studied in the present research paper are simultaneously {\it additive} and {\it abelian}; most of the basic notation and terminology are in agreement with the classical books \cite{F} and \cite{Fu}. In general, $G$ will denote an abelian group with torsion $T$, and if $p$ is a prime, $T_p$ will be its $p$-torsion subgroup; so $T=\oplus_p T_p$. As usual, all ``summand" means ``direct summand".  We will often use other standard conventions, such as $A\leq G$ to indicate that $A$ is a subgroup of $G$, $A\leq^\oplus G$ to mean $A$ is a summand of $G$,  ``rank'' for ``torsion-free rank'' and ``iff'' for ``if and only if.''

\medskip

A group $G$ is called {\it Hopfian} iff it is {\it not} isomorphic to any of its {\it proper} quotients; that is, each of its surjective endomorphisms is injective, i.e., an automorphism. Mimicking \cite{SH15}, a group $G$ is said to be {\it strongly Hopfian} (or simply {\it SH}) provided that,  for each endomorphism $\phi$ of $G$, the ascending chain $$\ker \phi\leq \ker \phi^2\leq\dots \leq \ker \phi^n \leq \dots$$ is eventually stationary, i.e., it stabilizes at some place.

If we let $G^n:=\phi^n(G)$, then it is readily checked that $G$ is SH iff, for all endomorphisms $\phi:G\to G$, there is an integer $n\in \N$ such that $\phi$ restricts to an injection on $G^n$; we call such an $n$ (if it exists) an \textit{SH-exponent} for $\phi$. If this holds, then $\phi$ also restricts to an injection on $G^m\leq G^n$ for all $m\geq n$, so that $\phi$ restricts to an isomorphism $G^m\cong G^{m+1}$ and $\ker \phi^n=\ker \phi^m$ for all $m\geq n$. In particular, if $G$ is SH and $\phi$ is surjective, then $G^n=G^{n+1}=G$, so that we can conclude that $\phi$ is an automorphism, i.e., $G$ is Hopfian.

We say $G$ is \textit{uniformly SH} if there is a fixed $n\in \N$ that is an SH-exponent for all endomorphisms $\phi:G\to G$. Obviously, if $G$ is uniformly SH, then it is SH.

Dually, if the ascending chain $\ker \phi^n$ is replaced by the descending chain $\phi^n(G)=G^n$, then we arrive at the notion of {\it strongly co-Hopfian} (or simply Sco-H) groups, as well as {\it uniformly} Sco-H. The difference between these two notions is well illustrated by the (easily checked) fact that a torsion-free group of finite rank will always be SH (in fact, uniformly SH), but it will be Sco-H iff it is divisible.

Characterizing any of the above four classes of groups quickly reduces to the case of reduced groups (see, for example, Corollary~\ref{reduced}). And if $G$ is reduced and a member of any of these four classes, it will follow that for any prime $p$, that $T_p$ will be finite, and in particular, it will be bounded (see, for example, Proposition~\ref{finite}). It follows from Theorems~2.1 from both \cite{SH15} and \cite{A} that if $G=T$ is reduced and torsion, then it will be in any of these classes iff there is a fixed exponent $e\in \N$ such that for all primes $p$, $T_p$ has at most $p^e$ elements (see also Corollary~\ref{torcase}).

Two of the most studied classes of mixed abelian groups are the \textit{sp-groups} and the \textit{(global) Warfield groups}, and these classes will be the focus of this paper, as well.
To say $G$ is an sp-group means that the inclusion $T=\oplus_p T_p\hookrightarrow \prod_p T_p=:P$ extends to a pure embedding $G\hookrightarrow P$ (i.e., $G/T$ is divisible). Again, when studying the groups in any of these classes, we may assume $G$ is reduced and each $T_p$ is finite, and in this case (by Proposition~\ref{spgroups}), $G$ will be an sp-group iff $G/T$ is divisible. This class is the focus of Section~\ref{secondsec}.

By \cite[Proposition~2.3]{CDK}, if $G$ is a reduced Sco-H group, then $G$ must be an sp-group. One of the main results of this paper (Theorem~\ref{general}) improves upon this by showing that \textit {any} reduced Sco-H (or uniformly Sco-H) group is necessarily SH (respectively, uniformly SH). When might a converse of this result hold, i.e., if $G/T$ is divisible, when can we conclude that $G$ is Sco-H? We verify this in two general cases: When $G$ is SH and $G/T$ has finite rank (Proposition~\ref{last}), and when $T$ is SH (Proposition~\ref{extending}).

Recall that a group is \textit{Warfield} if it is a direct summand of a group with a simple presentation. In our discussions, since we are concerned with groups whose $p$-torsions, $T_p$, are all finite (and hence bounded), these can be decomposed into groups of torsion-free rank 1; again, with finite $p$-torsion (see Proposition~\ref{warfdec}). We denote the class of such rank 1  groups with bounded $p$-torsion by $\B$ ; so any SH Warfield group will be a direct sum of groups in $\B$.  Clearly, any torsion-free group of rank 1 has a simple presentation, and is an element of $\B$.  It follows that the class of Warfield group that are torsion-free coincides with the completely decomposable groups.

The theory of Warfield groups makes extensive use of a category called \textit{WALK}, where two homomorphism $f,g:H\to G$ are considered \textit{WALK-equivalent} if $(f-g)(H)\leq G$ (see, for e.g., \cite[Section~15.8]{Fu}). This allows us to generalize the notion of the \textit{type} of a rank one torsion-free group to the \textit{WALK-type} of an element of $\B$. Using this, we are able to characterize exactly when a Warfield group $G$ whose torsion $T$ is SH also satisfies that property using the WALK-types of its summands in a given decomposition of $G$ into groups from $\B$ (Corollary~\ref{WarWalk}). And since a completely decomposable group is just a Warfield group with $T=0$ (which is obviously SH), this characteriation  generalizes the main result of \cite {CD} (see Corollary~\ref{compdectfg}).

Deciding when a general Warfield group is SH, when we are not given that $T$ is SH, is a much more difficult problem. We provide an explicit solution of only one case of that problem: when $G$ has rank 1 and each $T_p$ is cyclic (Proposition~\ref{warfieldrankone}). The difficulty of that computation is evidence that giving a complete characterization of all groups that are SH, especially those whose torsion is not SH, may ultimately be a problem of insurmountable difficulty.

\section{SH groups and sp-Groups}\label{secondsec}

Often, for some fixed endomorphism $\phi:G\to G$ and $n\in \N$, we will denote $\phi^n (x)=: x^n\in G^n$.  A sequential way to express the definition of SH groups is as follows:

\medskip

\noindent\textit{A group $G$ is SH iff, for each endomorphism $\phi:G\to G$, there is an $n \in \N$ such that for all $x\in G$, if $x^n\ne 0$, then $x^m\ne 0$ for all $m\geq n$.}

\medskip

Again, a value $n\in \N$ satisfying the above will be an SH-exponent for the endomorphism $\phi$, and a group $G$ is SH if, and only if, every map $\phi:G\to G$ has an SH-index. Clearly, if $n$ is an SH-exponent for $\phi$, then so is each natural number $n'\geq n$.

\medskip

The following two claims are easily verified:

\medskip\noindent(1.A) If $G$ is a torsion-free group of finite rank $n\in \N$, then $n$ is an SH-exponent for any endomorphism $\phi:G\to G$; so, $G$ is uniformly SH.

\medskip\noindent(1.B) Suppose $G=\oplus_{i\in I} G_i$, where each $G_i$ is fully invariant in $G$ and $\phi:G\to G$ is an endomorphism of $G$ whose restriction to $G_i$ is $\phi_i$. If $n\in \N$ is an SH-exponent for each $\phi_i$, then $n$ will also be an SH-exponent for $\phi$.

\medskip

We start our work with the following crucial observation, which actually is \cite[Proposition~2.7]{SH15} (we, however, include a brief proof for the convenience of the reader).

\begin{proposition}\label{finite}
If $G$ is an SH group, then each $T_p$ is finite.
\end{proposition}

\begin{proof}
Applying the definition with $\phi$ given by multiplication by $p$, then, for some $n\in \N$, it must be that $G[p^n]=G[p^{n+1}]$, which means that $T_p=G[p^n]$ is bounded. Therefore, $T_p\leq^\oplus G$, and if $T_p$ were infinite, then for some $m\leq n$ there would be a summand $K\leq^\oplus T_p\leq^\oplus G$ such that $K\cong \Z(p^m)^{(\omega)}$, which would contradict that $G$ is SH giving the claim.
\end{proof}

Remember that a group $G$ is {\it co-Hopfian} provided that it is {\it not} isomorphic to any of its {\it proper} subgroups, that is, each its injective endomorphism is surjective, i.e., an automorphism. In addition, the group $G$ is {\it strongly co-Hopfian} or just {\it Sco-H} for short, provided that the descending chain
$$G^1\geq G^2\geq\dots \geq G^n \geq \dots$$ is eventually stationary, i.e., for each endomorphism $\phi$ of $G$, it stabilizes at some place  (see, for example, \cite{A} or \cite{CDK}).

\medskip

The following assertion is just a restatement of Theorems~2.1 from both \cite{SH15} and \cite{A}.

\begin{corollary}\label{torcase}
Suppose $G$ is a reduced torsion group. Then, the following are equivalent:

(1) $G$ is SH.

(2) $G$ is uniformly SH.

(3) $G$ is Sco-H.

(4) $G$ is uniformly Sco-H.

(5) There is an integer $e\in \N$ such that, for all primes $p$, $G_p$ has cardinality at most $p^e$.
\end{corollary}

We note the following important fact, which parallels \cite[Theorem~2.13]{CDK}:

\begin{proposition}\label{uniformly}
If $G$ is a uniformly SH group, then $T$ satisfies the five conditions above; e.g., it is SH.
\end{proposition}

\begin{proof}
Suppose $n$ is a uniform SH-exponent for $G$. If $p$ is any prime, then since in conjunction with Proposition~\ref{finite} each $T_p$ is finite, there is a decomposition $G=T_p\oplus H_p$ (see \cite{F}). It thus follows that $n$ will also be a uniform SH-exponent for $T_p$. So, by point (1.B) listed above, $n$ will also be a uniform SH-exponent for $T=\oplus_p T_p$, as required.
\end{proof}

We will use the following  immediate consequence of Corollary~\ref{torcase} and Proposition~\ref{uniformly}

\begin{corollary}\label{torsum}
If $G=H\oplus K$, where $H$ and $K$ are uniformly SH, then $T$ is SH.
\end{corollary}

It is important to emphasize, that just because $G$ is SH, or Sco-H (but \textbf{not} uniformly so), it is possible for $T$ to fail to be SH (which in this case is the same as failing to be Sco-H). We will return to this point repeatedly. Again, if $G$ is reduced with one of these properties, it \textbf{will} follow that each $T_p$ is finite, just that they will not necessarily all be of cardinality at most $p^e$ for some fixed $e\in \N$.

\medskip

We proceed now by proving a few more technicalities.

\begin{lemma}\label{SH-exponent}
Suppose $H$ is a fully invariant subgroup of $G$, $\phi:G\to G$ is an endomorphism, $\hat \phi:H\to H$ is its restriction to $H$ and $\ov \phi:\ov G:=G/H\to \ov G$ is the induced endomorphism of $\ov G$. If $n_1$ is an SH-exponent for $\hat \phi$ and $n_2$ is an SH-exponent for $\ov \phi$, then $n=n_1+n_2$ is an SH-exponent for $\phi$.
\end{lemma}

\begin{proof}
Suppose $x\in G$, $x^n\ne 0$ and $m\geq n$; we need to show that $x^m\ne 0$. To that end, suppose first that $x^{n_2}\not\in H$. Therefore, since $n_2$ is an SH-exponent for $\ov \phi$ and $n_2\leq n\leq m$, we can infer that $x^m\not\in H$, so that, in particular, $x^m\ne 0$, as required.

If the last paragraph does not pertain, it follows that $y:=x^{n_2}\in H$. Now, since
\[
y^{n_1}=(x^{n_2})^{n_1}=x^n \ne 0
\]
and $m-n_2\geq n-n_2=n_1$, we can conclude that
\[
0\ne y^{m-n_2}=(x^{n_2})^{m-n_2}=x^m,
\]
completing the proof.
\end{proof}

The following statement generalizes \cite[Proposition~2.5]{SH15} and is an immediate consequence of Lemma~\ref{SH-exponent}.

\begin{corollary}\label{invariant}
Suppose $H$ is a fully invariant subgroup of $G$. If $G/H$ and $H$ are both SH (or uniformly SH), then $G$ is also SH (respectively, uniformly SH).
\end{corollary}

Notice that Corollary~\ref{invariant} has a number of immediate consequences like these:

\begin{corollary}
Suppose $G$ has torsion part $T$ such that both $T$ and $G/T$ are SH (or uniformly SH), then $G$ is also SH (respectively, uniformly SH).
\end{corollary}

\begin{corollary}\cite[Proposition~2.5]{SH15}\label{summandinv}
Suppose $G=H\oplus K$, where $H$ is fully invariant in $G$ (i.e., {\rm Hom}$(H,K)=0$). Then, $G$ is SH (or uniformly SH) iff both $H$ and $K$ are SH (or, respectively, uniformly SH).
\end{corollary}

\begin{proof}
The direction $(\Rightarrow)$ follows at once from the (fairly obvious) fact that SH groups are closed under summands (compare with \cite{SH15}). The direction $(\Leftarrow)$ follows directly from Corollary~\ref{invariant}.
\end{proof}

\begin{corollary}\label{reduced}
Suppose $G=D\oplus R$, where $D$ is divisible and $R$ is reduced. Then, $G$ is SH (or uniformly SH) iff both $D$ and $R$ are SH (or, respectively, uniformly SH).
\end{corollary}

\begin{corollary}
Suppose $G=T\oplus H$, where $T$ is torsion and $H$ is torsion-free. Then $G$ is SH (or uniformly SH) iff both $T$ and $H$ are SH (or, respectively, uniformly SH).
\end{corollary}

We will denote the {\it $p$-height} of $x\in G$ by $\vv x$, or simply by $\vvp x{}$ if there is no danger of confusion.

The group $G$ with torsion subgroup $T$ is said to be an \textit{sp-group} if $G\cong G'$, where $T\leq G'\leq P:= \prod_p T_p$, with $G'/G$ divisible. We will almost exclusively be concerned with the following special case of this concept.

\begin{proposition}\label{spgroups}
Suppose $G$ is a reduced group with bounded $p$-torsion (i.e., each $T_p$ is bounded). Then, $G$ is an sp-group iff $G/T$ is divisible.
\end{proposition}

\begin{proof}
Necessity being immediate from the definition of sp-groups, suppose that $G/T$ is divisible. It is easy to see that $P=\prod_p T_p$ will be the completion of $T$ in the $\Z$-adic topology. Since $T$ is dense in the $\Z$-adic topology on $G$, the inclusion $T\hookrightarrow P$ extends to a homomorphism $\phi:G\to P$; so, it will suffice to show that $\phi$'s kernel, which is the first Ulm subgroup, denoting it by $\widetilde G:=\{x\in G: \val x_p\geq \omega$ for all primes $p\}$, is divisible.

To that purpose, let $p$ be some arbitrary prime and suppose that $p^k T_p=0$. If $x\in \widetilde G$, then there is a $y\in G$ such that $\val y_p>k$ and $py=x$. If $\val y_p<\omega$, then we can find a $z\in G$ such that $\val y<\val z_p$ and $pz=x$. But this would imply that $y-z$ is a non-zero element of $T_p$ of height $\val y_p>k$, which cannot be since $p^k T_p=0$. So, $\val y_p\geq \omega$. Since multiplication by $p$ preserves all $q$-heights (where $q\ne p$ is another prime), we can infer that $\val y_q=\val x_q\geq \omega$. Therefore, $y\in \widetilde G$ so that $\widetilde G$ is $p$-divisible. Since $p$ was arbitrary, it follows that $\widetilde G$ is divisible; but, since $G$ is reduced, we can conclude that $\widetilde G=0$, as required.
\end{proof}

The following conclusion then comes from Propositions~\ref{finite} and~\ref{spgroups}.

\begin{corollary}
Suppose $G$ is a reduced SH group. Then $G$ is an sp-group iff $G/T$ is divisible.
\end{corollary}

By \cite[Proposition~2.3]{CDK}, if $G$ is a reduced Sco-H group, then $G$ must be an sp-group and each $T_p$ must be finite (and hence bounded). So, when discussing when a group with the SH property has the Sco-H one, or visa versa, there is no loss of generality in assuming that we are working with reduced sp-groups.

\medskip

We will use the following helpful observation more than once.

\begin{proposition}\label{spprop}
Suppose $G$ is an sp-group with bounded $p$-torsion. If $\gamma:G\to G$ is an endomorphism, then $G':=\gamma(G)$ is an sp-group with bounded $p$-torsion $G'\cap T=\gamma(T)$.
\end{proposition}

\begin{proof} Put $T':=\gamma(T)$; since $G'/T'$ is an epimorphic image of $G/T$, it follows that $G'/T'$ is divisible. Manifestly, $T'\leq T$ so that $T/T'$ is the torsion of $G/T'$. Since $T$ has bounded $p$-torsion, so does $T/T'$ and, in particular, $T/T'$ will be reduced. This forces that $(G'/T')\cap (T/T')=0$, so that $G'/T'$ will be torsion-free (and divisible). This yields that $G'\cap T=T'$, and thanks to Proposition~\ref{spgroups}, $G'$ will be an sp-group, as stated.
\end{proof}

We now want to discuss which reduced groups that are (uniformly) SH are also (uniformly) Sco-H, and visa versa. The next result shows that one implication is generally valid.

\begin{theorem}\label{general}
Suppose $G$ is a reduced group. If $G$ is Sco-H (or uniformly Sco-H), then $G$ is SH (or, respectively, uniformly SH).
\end{theorem}

\begin{proof} Again \cite[Proposition~2.3]{CDK} implies that $G$ must be an sp-group and each $T_p$ must be finite (and hence bounded).

Suppose $\phi:G\to G$ is an endomorphism and $\hat \phi$ is its restriction to $T$. For each $n\in \N$, set $G^n := \phi^n(G)$ and $T^n := \phi^n(T)$. So, Proposition~\ref{spprop} allows us to conclude that each $G^n$ is an sp-group with torsion $T^n$. Since $G$ is assumed to be Sco-H, for some $n\in \N$ we can conclude that $\phi$ restricts to a surjection $G^n\to G^n=G^{n+1}$. This insures that $T^n=T^{n+1}=\phi(T^n)$. But, since each $T^n_p$ is finite, this ensures that $\hat \phi$ is actually an automorphism of $T^n$.

If we now set $P^n:=\prod_p T^n$, it follows that $\hat \phi$ extends to an automorphism $\gamma$ on $P^n$; in particular, this $\gamma$ is injective. Note that $\phi$ and $\gamma$ agree on $T^n$, so they must agree on $G^n\leq P^n$, which guarantees that $\phi$ is also injective on $G^n$, as required.
\end{proof}

So, our question reduces to attempting to describe which reduced groups that are (uniformly) SH are also (uniformly) Sco-H. Again, since every (uniformly) Sco-H group is an sp-group (compare with \cite{CDK}), we may assume that $G/T$ is divisible. The following illustrates one easy case where this is an equivalence.

\begin{proposition}\label{last}
Suppose $G$ is a reduced group of finite rank. Then, $G$ is Sco-H (or uniformly Sco-H) iff $G$ is SH (or, respectively, uniformly SH) and $G/T$ is divisible.
\end{proposition}

\begin{proof} If $G$ is (uniformly) Sco-H, then we know it is an sp-group, so that $G/T$ is divisible. And by Theorem~\ref{general}, $G$ is (uniformly) SH.

Conversely, suppose $G$ is (uniformly) SH and $G/T$ is divisible. Referring to Propositions~\ref{finite} and~\ref{spgroups}, $G$ must be an sp-group. Let $\phi:G\to G$ be an endomorphism and, for all $n\in \N$, let $G^n=\phi^n(G)$. Since $G$ is SH, for some $n\in \N$ we can conclude that $\phi$ restricts to an injection $G^n\to G^n$. If, as usual, $T^n:=\phi^n(T)=G^n \cap T$, then each $T^n_p$ is finite and $G^n/T^n$ is divisible of finite rank. It now follows from \cite[Theorem~2.6]{K} that $G^n$ is co-Bassian, and hence co-Hopfian. Therefore, $\phi$ restricted to $G^n$ must actually be surjective, proving that $G$ will be (uniformly) SH, as expected.
\end{proof}

The following is obviously a restatement of Proposition~\ref{last}.

\begin{corollary}\label{obviously}
If $G$ is a reduced sp-group of finite rank, then $G$ is (uniformly) SH iff it is (uniformly) Sco-H.
\end{corollary}

The following important example shows that the torsion subgroup of a group with SH may not have that property.

\begin{example}\label{ptothep}
There is a reduced sp-group $G$ of rank 1 that is SH, but the torsion subgroup $T$ of $G$ is not SH.
\end{example}

\begin{proof} Basically, we can just let $G$ be the group from Example~2.15 of \cite{CDK}. That group was clearly reduced and had rank 1, so was of finite rank, and since it was an sp-group, the factor-group $G/T\cong \Q$ was divisible. Now, since in that work it was observed that $G$ is Sco-H, Theorem~\ref{general} or Proposition~\ref{last} or Corollary~\ref{obviously} all tell us that $G$ is also SH. And in view of Corollary~\ref{torcase}, the group $$T\cong \oplus_{n\in \N} \Z(p_n^n)$$ obviously fails to be SH, as asserted.
\end{proof}

We next aim to demonstrate that for reduced sp-groups of infinite rank, the class of SH groups is {\it properly} larger than the class of Sco-H groups. The following gives a simple way of constructing sp-groups that are uniformly SH.

\begin{proposition}\label{extending}
Suppose $G$ is a reduced sp-group. Then $G$ is uniformly SH iff $T$ is (uniformly) SH.
\end{proposition}

\begin{proof}
We may assume $T\leq G \leq P:= \prod_p T_p$ and $G/T$ is torsion-free divisible.

Further, assume first that $T$ is SH, so that it is uniformly SH in virtue of Corollary~\ref{torcase}. Let $n\in \N$ be an SH-index for all endomorphisms of $T$. To verify that $G$ is uniformly SH, suppose that $\phi:G\to G$ is an endomorphism; let $\hat \phi$ be its restriction to $T$. It thus follows that $\hat \phi$ is injective on $T^n:=\phi^n(T)$. Referring to Proposition~\ref{spprop}, if $G^n=\phi^n (G)$, then $G^n$ is also an sp-group. However, Proposition~\ref{finite} allows us to derive that each $T_p$ is finite, so that each $T^n_p$ is finite too. This assures that $\hat \phi$ is actually an automorphism of $T^n$.

Furthermore, as in the proof of Proposition~\ref{last}, if we let $P^n:=\prod_p T_p^n$, it follows that $\hat \phi$ extends to an automorphism $\gamma$ on $P^n$; in particular, this $\gamma$ is injective. Note that $\phi$ and $\gamma$ must agree on $G^n\leq P^n$, which means that $\phi$ is also injective on $G^n$, as required.

The converse follows directly from Proposition~\ref{uniformly}.
\end{proof}

The next construction substantiates our claim above:

\begin{example}\label{ingen} There is a reduced sp-group $G$ that is SH, but not Sco-H.
\end{example}

\begin{proof} In \cite[Example~20]{CK} an sp-group $G$ was constructed such that, for each prime $p$, $T_p$ was either 0 or isomorphic to $\Z(p)$. In particular, looking at Proposition~\ref{extending}, $G$ will be SH. On the other hand, it was observed there that this $G$ fails to be co-Hopfian, so that it also fails to be Sco-H, as required.
\end{proof}

We conclude this section with a couple of results on direct sums of (uniformly) SH groups.

\begin{proposition}
Suppose $G=H\oplus K$ is a reduced sp-group. Then $G$ is uniformly SH iff $H$ and $K$ are uniformly SH.
\end{proposition}

\begin{proof}
Necessity being clear, suppose $H$ are $K$ are uniformly SH. By Corollary~\ref{torsum}, $T$ will be SH. So, by Proposition~\ref{extending}, $G$ is uniformly SH.
\end{proof}

\begin{proposition}
Suppose $G=H\oplus K$, where $K$ is finite (and hence torsion). Then $G$ is SH iff $H$ is SH.
\end{proposition}

\begin{proof}
Necessity again being clear, suppose that $H$ is SH (and $K$ is finite). An elementary induction shows that it suffices to assume that $K$ is a $p$-group for some prime $p$. If $S$ is the torsion subgroup of $H$, then $S_p$ will also be finite, and hence a summand of $H$. Suppose $H=H'\oplus S_p$, so that $G=H'\oplus S_p\oplus K$. Now, $S_p\oplus K$ will be finite, so trivially SH.  In addition, the $p$-torsion of $G$, $S_p\oplus K$, is fully invariant in $G$. As $H'$ is a summand of $H$, and hence SH, by Corollary~\ref{summandinv}, $G$ will also be SH.
\end{proof}

\section{SH groups and (global) Warfield groups}

Recall that a group is  \textit{Warfield} if it is a summand of a simply presented group. In particular, it is clear that a summand of a Warfield group will be Warfield. In addition, as any torsion-free rank 1 group clearly has a simple presentation,  any totally decomposable torsion-free group is simply presented, and hence Warfield.

A perhaps more useful approach to Warfield groups involves the notion of a \textit{nice decomposition basis}, i.e., a subset $\{z_i\}_{i\in I}\subseteq G$ such that $Z=\oplus_{i\in I} \langle z_i\rangle$ is a free and nice subgroup of $G$ such that all $p$-height functions restricted to $Z$ respect the above decomposition. Then, $G$ is Warfield precisely when it has a nice decomposition base such that $G/Z$ is totally projective (see, for instance, \cite[Theorem~15.7.6]{Fu} as well as, for a more detailed information, we refer the interested reader to \cite{HR}). Fortunately, we are only concerned with Warfield groups with bounded $p$-torsion, where the deep theory of (global) Warfield groups is considerably more straightforward (compare with \cite{DK} as well).

\medskip

Recall that if $N\leq G$, $p$ is a prime and $\alpha$ is an ordinal, then
\[
              U_\alpha^p(N)=\{x\in N: \val x_p\geq \alpha \land \val {px}_p>\alpha+1\}
\]
is the \textit{$\alpha$th $p$-Ulm factor} of $N$ in $G$. There is a natural injection $U_\alpha^p (N)\to U_\alpha^p(G)$, and we say $N$ is \textit{tight} in $G$ if this is an isomorphism for all primes $p$ and ordinals $\alpha$ (compare also with \cite{DK}). This is equivalent to saying that all of the relative Ulm invariants of $N$ in $G$ are 0. Again, since we are only concerned with groups with bounded $p$-torsion, our ordinal $\alpha$ need only take on finite values.

\medskip

In what follows, let $\B$ be {\it the class of all groups of rank 1 with bounded $p$-torsion}. Notice the pretty simple observation, as shown below, that every $G\in \B$ is Warfield.

\medskip

The following statement is useful for our further discussions.

\begin{proposition}\label{warfdec}
If $G\in \B$, then $G$ is Warfield. In addition, if $G$ is an arbitrary Warfield group with bounded $p$-torsion, then either $G$ is torsion or $G=\oplus_{i\in I} G_i$, where each $G_i\in \B$.
\end{proposition}

\begin{proof}
Let $G\in \B$. If $\langle z\rangle =Z\leq G$ is an infinite cyclic subgroup, then it is well known that $Z$ must be nice in $G$. For every prime $p$, the $p$-torsion-subgroup of $G/Z$ will be the direct sum of a divisible group and a bounded group, and hence totally projective. It, therefore, follows that $z$ alone constitutes a nice decomposition basis for $G$, proving that it is Warfield.

Suppose now that $G$ is any Warfield group with bounded $p$-torsion; we may clearly assume that $G$ is not torsion. Let $\{z_i\}_{i\in I}$ be a nice decomposition basis for $G$. For each $i\in I$, it is pretty clear that we can construct a group $A_i\in \B$ containing $Z_i:=\langle z_i\rangle$ such that the $p$-height functions on $A_i$, restricted to $Z_i$, agree with the original $p$-height functions from $G$ on $Z_i$, and for which $Z_i$ is a tight subgroup of $A_i$ (for example, if $Z_i\leq H\leq G$ is so that $H/Z_i$ is the torsion subgroup of $G/Z_i$, then $G/H$ is torsion-free and hence $H$ is isotype in $G$. Clearly, $T$ is also the torsion of $H$, which has rank 1. It now follows from
\cite[Theorem~1]{K} that $H=A_i\oplus T'$, where $Z_i$ is tight in $A_i$). Let $S$ be a torsion group with bounded $p$-torsion such that the Ulm invariants of $S$ all agree with the relative Ulm invariants of $Z=\oplus_{i\in I} \langle z_i\rangle$ in $G$. We have set things up so that the identity function $Z\to Z$ extends to an isomorphism $$G\to (\oplus_{i\in I} A_i)\oplus S,$$ giving our desired decomposition.
\end{proof}

If $G$ and $G'$ are in $\B$, we write $G\precsim G'$ if there is a homomorphism $\phi:G\to G'$ such that $\phi(G)$ is not contained in the torsion subgroup $T'$ of $G'$. Note that this last condition is equivalent to requiring that the induced homomorphism, $\ov \phi:G/T\to G'/T'$, is non-zero and, in fact, if it is non-zero, it plainly must be injective. In particular, this means that $\phi:G\to G'$ is witness to $G\precsim G'$ exactly when the kernel of $\phi$ is contained in $T$. In other language, we are requiring that $\phi$ is non-zero in the category {\it WALK} (see, for example, \cite[Section~15.8]{Fu}).

It follows that, if $G\precsim G'$, then the (ordinary) types $\tau$ and $\tau'$ of $G/T$ and $G'/T'$ satisfy $\tau\leq \tau'$. However, suppose $G=\Q$ and $G'$ is any reduced sp-group of rank 1 with bounded $p$-torsion. Then, $$G/T\cong G'/T'\cong \Q,$$ but, since $G'$ is reduced, there is no (non-zero) homomorphism $G\to G'$; so the converse does not hold.

\medskip

We will say now that $G$ and $G'$ have the same \textit{WALK-type} if $G\precsim G'\precsim G$, and when this happens, we write $G\approx G'$. It is fairly clear that $\approx$ is an equivalence relation on $\B$. If $G, G'\in \B$ and $G\approx G'$, it is also clear that the rank 1 torsion-free groups $G/T$ and $G'/T'$ have the same type (regarded in the usual sense).

If $G$ is a group and $x\in G$, then the \textit{height matrix} of $x$, written as ${\mathbb H}_G(x)$ has, for each prime $p$, a row consisting of the $p$-height sequence of $x$. If $ m\in \N$, then the equality $m{\mathbb H}_G(x)={\mathbb H}_G(mx)$ holds always.

\medskip

The following statement gives us a concrete description of when elements of $\B$ have the same WALK-type as defined above.

\begin{proposition}
Suppose $G, G'\in \B$ (so they have rank 1). Then, the following four statements are equivalent:

(1) $G\approx G'$.

(2) If $z\in G$ and $z'\in G'$ are non-torsion elements, then
\[
m{\mathbb H}_G(z) = m'{\mathbb H}_{G'}(z')
\]
for some $m, m'\in \N$.

(3) $G\cong A\oplus S$ and $G'\cong A'\oplus S'$, where $A\cong A'$ have rank 1 and tight infinite cyclic subgroups, say $Z$ and $Z'$ (so that both $S$, $S'$ are torsion).

(4) $G\oplus B\cong G'\oplus B'$ for some torsion groups $B$ and $B'$ with bounded $p$-torsion.
\end{proposition}

\begin{proof}
\noindent(1) $\Rightarrow$ (2): Suppose $\alpha:G\to G'$ and $\alpha': G'\to G$ bear witness to $G\precsim G'\precsim G$. Let $Z=\langle z\rangle \leq G$ and $Z':= \langle z'\rangle \leq G'$. Replacing $\alpha$ and $\alpha'$ by some multiples, we may assume $\alpha(Z)\leq Z'$ and $\alpha'(Z')\leq Z$. Suppose also $\alpha(z)=k'z'\in  Z'$ and $\alpha'(z')=kz\in Z$. It follows that
\[
{\mathbb H}_{G}(z)\leq  {\mathbb H}_{G'}(\alpha(z))= k'{\mathbb H}_{G'}(z'))\leq  k'{\mathbb H}_{G}(\alpha'(z'))
= kk'{\mathbb H}_{G}(z)).
\]
Observe that, if $p$ is not a divisor of $k$ or $k'$, then this implies that $z$ and $z'$ have the same sequence of $p$-heights. And for each prime divisor $p$ of either $k$ or $k'$, the $p$-height sequences of $z$ and $z'$ eventually look like either $(m, m+1, m+2, \dots)$ (where $m\in\N$) or $(\infty, \infty, \infty, \dots)$; in particular, after some initial segments, these $p$-height sequences must be equal. Since there are a finite number of such primes, it is readily seen that (2) is true.

\medskip

\noindent(2) $\Rightarrow$ (3): Replacing $z$ by $mz$ and $z'$ by $m'z$, there is no loss of generality in assuming that ${\mathbb H}_G(z) ={\mathbb H}_{G'}(z')$. It follows that $z\to z'$ extends to a $p$-height-preserving isomorphism $Z\to Z'$. It also follows from \cite[Theorem~1]{K} that $G=A\oplus S$, where $Z$ is tight in $A$ and, similarly, $G'\cong A'\oplus S'$. Since $Z$ and $Z'$ will be nice subgroups of $A$ and $A'$ respectively, all of whose relative Ulm invariants are 0, we can deduce that $Z\cong Z'$ extends to $A\cong A'$, as required.

\noindent(3) $\Rightarrow$ (4 ): Letting $B=S'$ and $B'=S$, we have
\[
G\oplus B=(A\oplus S)\oplus S'\cong (A'\oplus S')\oplus S=G'\oplus B'.
\]

\noindent(4) $\Rightarrow$ (1): It straightforwardly follows that the composition, say $\gamma$, of the natural homomorphisms
\[
G\to G\oplus B\cong G'\oplus B'\to G',
\]
 after modding out by their respective torsion subgroups, becomes a sequence of isomorphisms
\[
G/T\cong (G \oplus B)/(T\oplus B)\cong (G'\oplus B')/(T'\oplus B')\cong G'/T'.
\]
In particular, we cannot have $\gamma(G)\leq T'$, so that $G\precsim G'$. The reverse relation is analogous, ensuring the proof.
\end{proof}

The following criterion gives a complete description of the uniformly SH Warfield groups.

\begin{theorem}\label{warfieldchar}
Suppose $G=\oplus_{i\in I} G_i$ is a Warfield group with torsion subgroup $T$ such that each $G_i\in \B$, and let $T_i$ be the corresponding torsion subgroup of $G_i$ so that $T=\oplus_{i\in I} T_i$.

(1) If $G$ is SH, then there is an $n\in \N$ such that there is no subset $\{i_1, i_2, \dots, i_{n-1}, i_{n}\}$ of $I$ with
$n$-elements such that
\[
 G_{i_1}\precsim G_{i_2}\precsim  \cdots \precsim G_{i_{n-1}}\precsim G_{i_{n}}.
 \]

(2) $G$ is uniformly SH iff there is a positive integer $n$ as in (1) and $T$ is SH.
\end{theorem}

\begin{proof}
Considering (1), we argue by contrapositive, so assume the ordering condition fails. Put $$L:=\{(j,k)\in \N^2: j\leq k\}.$$ We claim that there is a subset $\{G_{j,k}: (j,k)\in L\}$ of (distinct) elements of $\{G_i:i\in I\}$ such that, if $1\leq j<j'\le k$, then $G_{j,k}\precsim G_{j',k}$: To begin, let $G_{1,1}$ be arbitrary. Suppose $k>1$ and we have chosen $G_{j,k'}$ for all $(j,k')\in L$ with $k'<k$ satisfying our conditions. By our hypothesis, we can find $$n=k(k+1)/2=1+2+\cdots +k$$ elements of the $G$s that are totally ordered by the operation $\precsim$. Evidently, at least $k$ of these elements do not appear in the list of $G_{j, k'}$ with $k'<k$ chosen previously. Labeling these $k$ elements by $(j,k)$ with $1\leq j\leq k$ in ascending order using $\precsim$ completes our inductive definition.

Define $\phi:G\to G$ as follows: If $(j,k)\in L$ with $j<k$, let $\phi$ on $G_{j,k}$ be a homomorphism that bears witness to $G_{j,k}\precsim G_{j+1,k}$. We also define $\phi$ to be 0 on each $G_{k,k}$, as well as on any $G_i$, where $i$ is an index that does not appear in the list $G_{j,k}$ for $(j,k)\in L$.

Furthermore, for any $k$, let $z_k\in G_{1,k}$ be a non-torsion element. It is readily seen that $z_k^{k-1}\ne 0$, but $z_k^k=0$. This gives that $G$ is not SH, as required.

\medskip 

For point (2), the implication $(\Rightarrow)$ follows from (1) and Proposition~\ref{uniformly}, so assume $T$ is SH and $n$ is as in (1). As $T$ will also be uniformly SH, there is an $m\in \N$ which is a uniform SH-exponent for $T$.

Clearly, the existence of such an $n$ implies that, if $I'\subseteq I$ is non-empty, then $\{G_i: i\in I'\}$ will have a maximal element under the relation $\precsim$.

Let $I_1\subseteq I$ be the collection of indices of those summands that are maximal with respect to $\precsim$. Next, let $I_2\subseteq I\setminus I_1$ be the collection of indices of the maximal elements of the remaining groups, again under the relation $\precsim$. Note that, if $i\in I_2$, then for some $i'\in I_1$ we must have $G_{i}\prec G_{i'}$. Continuing to define $I_3, I_4, \dots$ in this way, the last sentence tells us that it must terminate after at most $n$ stages, i.e., $I$ will be the disjoint union $I_1\cup \cdots \cup I_n$.

Apparently, defining $i\sim i'$ to mean $G_i \approx G_{i'}$ gives an equivalence relation on $I$, and each equivalence class is contained in precisely one $I_j$ and has at most $n$ elements. Letting $$\mathcal I_j:=\{[i] ~ | ~ i\in I_j\}$$ be the corresponding partition of $I_j$ into equivalence classes, again each such equivalence class in $\mathcal I_j$ will have at most $n$ elements.

Let $K_0:=\emptyset$ and, for $j=1, \dots, n$, let $$K_j:=I_1\cup \cdots \cup I_j.$$ And for $j=0, 1, \dots, n$, let $$H_j:=T+\oplus_{i\in K_j} G_i;$$ so, in particular, $H_0=T$. Obviously, $H_0=T$ is fully invariant in $G$, and we assert that, if $j=1, 2, \dots, n$, then $H_j$ is, as well: This basically follows since if $j<j'$, and $i\in I_j$ and $i'\in I_{j'}$, then we cannot have that $G_{i}\precsim G_{i'}$, so that any homomorphism $\sigma:G_i\to G_{i'}$ must satisfy $\sigma (G_i)\leq T_{i'}$, i.e., $\sigma$ is 0 in the category WALK.

We next prove by induction, for $j=0, 1, \dots, n$, that $H_j$ has the number $m+jn$ as a uniform SH-exponent. This is true by definition for $j=0$, so assume it holds for $j-1<n$. So, let $\phi$ be any endomorphism of $H_j$. Note that $\phi$ induces an endomorphism $$\ov \phi:\ov H:= H_j/H_{j-1}\to \ov H$$ and, in view of Lemma~\ref{SH-exponent}, it suffices to show that $n$ is an SH-exponent for $\ov \phi$.

Note that $\ov H \cong \oplus_{\alpha\in \mathcal I_j} C_\alpha$, where each $C_\alpha=\oplus_{i\in \alpha} (G_i/T_i)$ is torsion-free of rank at most $n$. If $\alpha, \alpha'$ are distinct elements of $I_j$, and $i\in \alpha$, $i'\in \alpha'$, then $G_i$ and $G_{i'}$ are not comparable via the operation $\precsim$, meaning that if $\gamma:G_i\to G_{i'}$, then we must have that $\gamma(G_i)\subseteq T_{i'}$. It thereby follows that $\ov \phi(C_\alpha)\leq C_{\alpha}$ for each $\alpha\in \mathcal I_j$. Exploiting now point (1.A) quoted above, this restriction has SH-exponent $n$, so by point (1.B) alluded above, the map $\ov \phi$ also has SH-exponent $n$, thus finishing the proof.
\end{proof}

We now note some variations on the last result.

\begin{corollary}\label{WarWalk}
Suppose $G =\oplus_{i\in I} G_i$ is a Warfield group, where each $G_i\in \B$ and $T$ is SH. Then the following are equivalent:

(1) $G$ is uniformly SH;

(2) $G$ is SH;

(3) There is an $n\in \N$ such that there is  no subset $\{i_1, i_2, \dots, i_{n-1}, i_{n}\}$ of $I$ with
$n$-elements such that
\[
 G_{i_1}\precsim G_{i_2}\precsim  \cdots \precsim G_{i_{n-1}}\precsim G_{i_{n}}.
 \]
\end{corollary}

\begin{proof}
(1)$\Rightarrow$(2) is obvious, (2)$\Rightarrow$(3) is by Theorem~\ref{warfieldchar}(1) and  (3)$\Rightarrow$(1)
is by Theorem~\ref{warfieldchar}(2).
\end{proof}

Again, in the torsion-free case, the Warfield groups agree with the completely decomposable groups, and what we have termed the WALK-type of a rank 1 torsion-free group $G$ agrees with its type, $\tau(G)$ in the usual sense. So, if $T=0$ in the last result (which is trivially SH), we then arrive at the following assertion, which is a restatement of the main result of \cite{CD}:

\begin{corollary}\label{compdectfg}
Suppose $G$ is a completely decomposable torsion-free group. Then the following are equivalent:

(1) $G$ is uniformly SH;

(2) $G$ is SH;

(3) There is an $n\in \N$ such that there is no subset $\{i_1, i_2, \dots, i_{n-1}, i_{n}\}$ of $I$ with
$n$-elements such that
\[
 \tau(G_{i_1})\leq \tau(G_{i_2})\leq  \cdots \leq \tau(G_{i_{n-1}})\leq \tau( G_{i_{n}}).
 \]
\end{corollary}

 We finish this section with an application on direct sums and summands of uniformly SH Warfield groups.

\begin{proposition}
Suppose $G=H\oplus K$ is a Warfield group. Then $G$ is uniformly SH iff $H$ and $K$ are both uniformly SH.
\end{proposition}

\begin{proof}
Necessity being obvious, suppose $H$ and $K$ are uniformly SH. By Corollary~\ref{torsum}, $T$ is SH.

According to \cite{HR}, both $H$ and $K$ are also Warfield groups.
Suppose $H=\oplus_{i\in I} H_i$ and $K=\oplus_{j\in J} K_j$, where each $H_i$ and $K_j$ are in $\B$. Suppose also $n_H$ and $n_K$ are as in the statement of Corollary~\ref{WarWalk}, witnessing that $H$ and $K$, respectively, are SH. If we let $n=n_H+n_K$, it is not too hard to verify that the condition of Corollary~\ref{WarWalk} is satisfied for $G$ using this value $n$, so that $G$ is also uniformly SH, as asked for.
\end{proof}

\section{More Warfield groups}

Deciding when a Warfield group is SH, as opposed to the uniformly SH Warfield groups described in Corollary~\ref{WarWalk}, appears to be a decidedly more difficult problem. The following non-trivial argument handles just one simple case of this question.

\begin{proposition}\label{warfieldrankone}
Suppose $G\in \B$ (is a Warfield group of rank 1) with a tight subgroup $Z=\langle z\rangle$ of rank 1 such that each $T_p$ is cyclic. For any prime $p$, let $T_p\cong \Z(p^{e_p})$ and if $e_p\ne 0$, let $j_p>0$ be the size of the single jump in the $p$-height sequence of $z$, and suppose that jump takes place between $\vv {p^{k_p-1}z}$ and $\vv {p^{k_p}z}$. So, the $p$-height sequence of $z$ satisfies
$$
           m_p:=\vv z< \cdots<m_p+k_p-1=\vv {p^{k_p-1}z}<m_p+k_p+j_p=\vv {p^{k_p}z}.
$$
Then $G$ is SH iff there is an $M\in \N$ such that $e_p\leq M\min\{j_p, k_p\}$ for all primes $p$.
\end{proposition}

\begin{proof} We begin our argument by giving an explicit description of what this $G$ must look like by giving it a specific simple presentation.

Let $\ell_p:=\min\{j_p, k_p\}$ for all primes $p$. In addition, let $\sigma_p= e_p+j_p+1$ when $j_p\in \N$, and if $j_p=\infty$, let $\sigma_p=\infty$. Let $x_p\in G$ satisfy $p^{m_p} x_p=z$, and let $\{y^i_p\}_{0<i<\sigma_p}\subseteq G$ be such that $$py^1_p= p^{k_p} z=p^{m_p+k_p} x_p$$ and $py^{i}_p=y^{i-1}_p$ whenever $1<i<\sigma_p$. Note that $e_p=m_p+k_p$ and $T_p=\langle x_p- y_p^{e_p}\rangle$.

Suppose $G$ is SH; we need to find an integer $M>0$ such that $e_p\leq M \ell_p$ for all primes. In doing so, we will define a homomorphism $\phi:G\to G$ such that $\phi(z)=0$, so that $\phi(G)\leq T$. Since $G$ is SH, there is an $M\in \N$ such that $\textrm{ker}(\phi^M)= \textrm{ker}(\phi^{M+1})$, and we will show that this $M$ must satisfy the indicated inequality. We construct $\phi$ by defining it one prime at a time, specifying each $\phi(x_p)$ and $\phi(y^i_p)$.

\medskip

We consider two possible cases:

\medskip \noindent\textit{Case 1} -- $j_p\leq k_p<\infty$, so that $\ell_p=j_p$: In this case we set
$$
         \phi(x_p)=0\ \ \mathrm{and}\ \ \phi(y^{e_p+j_p}_p) =y^{e_p}-x_p\in T_p\cong \Z(p^{e_p}).
$$
It follows that
\[
\phi(x_p-y^{e_p}_p)=\phi(x_p)-\phi(p^{j_p} y^{e_p+j_p})=p^{j_p} (x_p-y^{e_p}_p),
\] i.e., $\phi$ restricts to multiplication by $p^{j_p}$ on $T_p$. Since $T_p$ is cyclic, it follows that
$$
         0=    \phi^M(x_p-y^{e_p}_p)= p^{M\cdot j_p} (x_p-y^{e_p}_p).
             $$
This gives that $e_p\leq M\cdot j_p=M\cdot \ell_p$, as asserted.

\medskip \noindent\textit{Case 2} -- $k_p\leq j_p$, so that $\ell_p=k_p$: In this case, we set
%Consider the homomorphism $\phi:G\to G$ such that
$$
         \phi(x_p):=p^{k_p}(x_p-y^{e_p}_p)\in T_p\cong \Z(p^{e_p})\ \ \mathrm{and}\ \ \phi(y^{i}_p) =0\ \ \mathrm{for\ all}\ i>0.
$$
Note that
\[
\phi(z)=\phi(p^{m_p}x_p)=p^{m_p+k_p}(x_p-y^{e_p}_p)=p^{e_p}(x_p-y^{e_p}_p)=0,
\]
as required.

It now follows that $\phi(x_p-y^{e_p}_p)=p^{k_p} (x_p-y^{e_p}_p)$, i.e., $\phi$ restricts to multiplication by $p^{k_p}$ on $T_p$. Again, since $T_p$ is cyclic, it follows that
%we can find an $M$ such that
$$
            0= \phi^M(x_p-y^{e_p}_p)= p^{M\cdot k_p} (x_p-y^{e_p}_p).
$$
This forces that $e_p\leq M\cdot k_p=M\cdot \ell_p$, as claimed.

Conversely, suppose we are given such an integer $M$; we want to show $G$ is SH. To that goal, let $\phi:G\to G$ be an endomorphism with kernel $K$.

Suppose first that $K\cap Z=0$, i.e., $K\leq T$. It thus follows that there is an $N$ such that, for all primes $p>N$, if we consider the endomorphism of $p$-localizations $\phi_{(p)}:G_{(p)}\to G_{(p)}$, we can conclude that $\phi_{(p)}$ restricted to $Z_{(p)}$ is an automorphism of $Z_{(p)}$. It thereby follows that, for all $p>N$ and $k<\omega$, the map $\phi$ induces an automorphism $U_k^p(G)\to U_k^p(G)$. However, since each $T_p$ is finite, it follows that $\phi$ restricts to an automorphism of $T_p$. If $S=\oplus_{q<N} T_q$, which is finite, we can infer that $K\leq S$. Similar reasoning shows that, if $K^j$ is the kernel of $\phi^j$, then $K^j\leq S$. However, since $S$ is finite and $K^j\leq K^{j+1}$, it follow that for some $j<\omega$ we have $K^j=K^{j+1}$, as desired.

So, we may assume that $K\cap Z\ne 0$, i.e., $\phi(G)\leq T$ (it is here that we need to use our $M$).

We claim that it will suffice to show that, for almost all primes $p$, we have $T_p^M=0$: Suppose we have verified this. It then follows that $G^{M+1}\leq T^M$ must be finite. But this informs us for some $n\geq M+1$ that $\phi$ must be an isomorphism on $G^n$.

Now, to verify for almost all primes $p$ that $T_p^M=0$, since each $T_p$ is cyclic, it suffices to check that for almost all primes $p$ we must have
$$
        \phi(G[p^{\ell_p}])=0.
$$
In fact, let $\mathcal F$ be the (finite) set of primes such that $\phi(z)\in \oplus_{p\in \mathcal F} T_p$. We will inspect now the above condition whenever $p$ is not in $\mathcal F$. In fact, localizing at some $p\not\in \mathcal F$, we may assume that $G=G_{(p)}$ and $\phi(z)=0$.

\medskip

We again break into two distinguish cases:

\medskip

\noindent\textit{Case 1} -- $j_p\leq k_p<\infty$: It follows that $G[p^{\ell_p}]=G[p^{j_p}]$ is generated by
$p^{k_p-j_p}z-y_p^{j_p}$. We know that $\phi(z)=0$, so that $\phi(p^{k_p-j_p}z)=0$. In addition,
$y_p^{j_p}=p^{e_p} y_p^{e_p+j_p}$ which ensures that $\phi(y_p^{j_p})=0$, completing this case.

\medskip

\noindent\textit{Case 2} -- $k_p\leq j_p$: It follows that $G[p^{\ell_p}]=G[p^{k_p}]$ is generated by
$z-y_p^{k_p}$. We know that $\phi(z)=0$. Since $e_p+k_p\leq e_p+j_p$, we obtain
$y_p^{k_p}=p^{e_p} y_p^{e_p+k_p}$. Thus, we can deduce that $\phi(y_p^{k_p})=0$, completing this case, and hence the whole proof.
\end{proof}

The group of \cite[Example~2.15]{CDK}  mentioned in the above Example~\ref{ptothep} fits the pattern described in  Proposition~\ref{warfieldrankone}. To see that, let $p_1, p_2, p_3, \dots$ be the collection of all primes. In $P:= \prod_{n\in \N} \Z(p_n^n)$, we let $z$ be the vector whose $n$th coordinate is $1\in \Z(p_n^n)$. Define the group $G$ as that which satisfies $T\leq G\leq P$ with $G/\langle {z}\rangle$ being the torsion subgroup of $P/\langle {z}\rangle$. So, for all $n\in\N$, the $p$-height sequence of ${z}$ is given by
\[
(0, 1, 2, \cdots, n-1, \infty, \infty, \infty, \dots)
\]

In the notation of Proposition~\ref{warfieldrankone}, for each prime $p_n$ ($n\in \N$), we have $m_n=0$, $k_n=e_n=n$ and $j_n=\infty$. Therefore, $\ell_n=\min\{j_n, k_n\}=n$, so that if $M=1$, then $e_n=n\leq 1\cdot n=M\cdot \ell_n.$  So, again utilizing Proposition~\ref{warfieldrankone}, this $G$ is SH, but its torsion part is obviously not SH, as expected.

\bigskip

\noindent{\bf Declarations:} The authors declare {\bf no} any conflict of interests as well as {\bf no} data was used while preparing and writing the current manuscript.

\end{document}